\documentclass{amsart}
\usepackage{amscd,amssymb,graphicx,MnSymbol}
\usepackage{tikz}
\usepackage{bm}
\newcommand{\of}[1]{\ensuremath{\!\left({#1}\right)}}

\tikzset{jumpSingle/.style={thick,->}}
\tikzset{jumpFamily/.style={thick,dashed}}
\tikzset{family/.style={very thick}}
\tikzset{jumpSingle/.style={thick,->}}
\tikzset{jumpFamily/.style={thick,dashed}}
\tikzset{family/.style={very thick}}
\tikzset{algebra/.style={inner sep=0pt}}
\tikzset{algebra/.style={inner sep=0pt}}

\input{xy}
\xyoption{all}

\author[Fialowski]{Alice Fialowski}
\address{
Alice Fialowski\\
University of P\'ecs and
E\"otv\"os Lor\'and University\\ Hungary
}
\email{fialowsk@ttk.pte.hu, fialowsk@cs.elte.hu}

\author[Penkava]{Michael Penkava}
\address{
Michael Penkava\\
University of Wisconsin-Eau Claire\\
Eau Claire, WI 54702-4004} \email{penkavmr@uwec.edu}

\thanks{This research was supported by grants from the University of Wisconsin-Eau Claire. The final version of the paper was written during the stay of the first author at the Max-Planck-Institute f\''ur Mathematik Bonn.}%

\newtheorem{thm}{Theorem}[section]

\theoremstyle{definition}
\newtheorem{dfn}[thm]{Definition}

\def\R{\hbox{$\mathbb R$}}

\def\C{\hbox{$\mathbb C$}}
\def\tens{\otimes}

\def\GL{\mathbb{GL}}
\def\gl{\mathfrak{gl}}
\def\sl{\mathfrak{sl}}
\def\P{\mathbb P}

\def\ra{\rightarrow}

\def\ainf{\mbox{$A_\infty$}}
\def\linf{\mbox{$L_\infty$}}
\def\tns{\otimes}
\def\ph{\varphi}

\def\k{\mathbb K}

\def\inv{^{-1}}

\def\s#1{(-1)^{#1}}
\def\mcom{,\dots,}
\def\mplus{+\cdots+}

\def\zt{$\mathbb{Z}_2$}

\def\im{\operatorname{Im}}

\def\diag{\text{diag}}

\begin{document}
\setlength{\multlinegap}{0pt}
\title[Stratification of moduli spaces]{Stratification of moduli spaces of Lie algebras, similar matrices and bilinear forms}%

\address{}%
\email{}%

\thanks{}%
\subjclass{}%
\keywords{}%
\date{\today}
\begin{abstract}
In this paper, the authors apply a stratification of moduli spaces of complex Lie algebras to analyzing the moduli spaces of $n\times n$ matrices under scalar similarity and  bilinear forms under the cogredient action. For similar matrices, we give a complete description of a stratification of the space by
some very simple projective orbifolds of the form $\P^n/G$, where $G$ is a subgroup of the symmetric group $\Sigma_{n+1}$ acting on $\P^n$ by permuting the projective coordinates.  For bilinear forms, we give a similar stratification up to dimension 4.
\end{abstract}
\maketitle
\section{Introduction}
The authors have been studying moduli spaces of algebras over the complex numbers for a long time, beginning with the construction of moduli spaces of low dimensional \linf\ algebras (see, for example \cite{fipe}). In studying \linf\ and \ainf\ algebras, a \zt-grading plays an important role, but the classical picture of algebras
without any \zt-grading fits into this picture as well.

When analyzing \linf\ algebras on a three dimensional space, we first realized that they are just the ordinary 3-dimensional Lie algebra structures on the space. Because we arrived at our classification of the algebras through an approach that focused on deformations of the algebras, we arrived at
a decomposition into strata that had some important differences with the classical decomposition of the space (see \cite{jac}).  Eventually, we discovered that moduli spaces of low dimensional ordinary and \zt-graded complex Lie and associative algebras have a decomposition into
strata consisting of some very simple types of projective orbifolds, of the form $\P^n/G$, where $G$ is a subgroup of the symmentric group $\Sigma_{n+1}$, which acts on $\P^n$ by permuting the projective coordinates.

Based on our construction of many different such moduli spaces, we conjectured that this type of decomposition happens for all such moduli spaces of finite dimensional algebras over $\C$ (see \cite{fp15,fp20}). We have verified this conjecture in many cases, but have so far not been able to establish it in general.  in this paper, we give an explicit construction of a stratification of part of the moduli space of Lie algebras of a given dimension in exactly this form, which holds in any finite dimensional space.

This part arises when considering Lie algebras which arise as extensions of a 1-dimensional (trivial) Lie algebra by a trivial $n$-dimensional Lie algebra.
These algebras are classified by the action of $\C^*\times\GL(n,\C)$ on the space $\gl(n,\C)$ of $n\times n$ matrices by conjugation and multiplication by
a scalar, which is sometimes called scalar similarity.  We will give a decomposition of the space of equivalence classes of matrices under this action into
strata that are parameterized by projective orbifolds.  Moreover, the deformations of the elements can be read directly from the forms of the matrices.
The classification is related to Jordan decomposition, but is actually coarser, because several classes of Jordan forms combine into one stratum. (In fact, there are a few other differences as well.)

Later, we were asked to compare the algebraic deformation theory with the analytic deformation theory, in particular, to relate the algebraic notion of a
miniversal deformation to the analytic one.  While these definitions are quite different, the relationship is very close, so we find that our ideas also
can be used to stratify moduli spaces arising in the analytic context from the same group action. Arnold in \cite{arn} gave a construction of a versal deformation of the moduli space of similar matrices based on their Jordan decomposition. Although the action he considered was similarity, rather than
scalar similarity, there is a direct parallel between his analysis and our point of view. Classification of similar matrices was first studied in \cite{arn}, but has been revisited and improved upon by, for example
\cite{horn-ser,dfs,eek}.

In the construction of moduli spaces, this time of associative algebras, 
we discovered that a part of the moduli space is described by the cogredient action of $\GL(n,\C)$ on
$\gl(n)$, in other words the action given by $G.A=GAG^*$. This moduli space is the space of equivalence classes of bilinear forms on $\C^n$.
 Bilinear forms over a field of characteristic not equal to 2 were classified
by  Riehm \cite{riehm} and Gabriel \cite{gabriel}. The classification in \cite{horn-ser} turns out to be very simple, because bilinear forms are either decomposable, or have an
explicit description.
In this paper,  we only give a stratification of the space of bilinear forms over $\C$ up to dimension 3, but already here we discovered that the natural stratification
requires a new way of decomposing the moduli space.

In our context, we studied only spaces up to dimension 4, and found that this moduli space also has a natural stratification by
projective orbifolds of exactly the same type.  We have not determined a classification of bilinear forms in general in terms of such a decomposition, but believe that such a classification exists. Our classification is quite different than the one in \cite{horn-ser}, and we will later give a comparison. It should
be noted that the projective structure does not require any multiplication by an element in $\C^*$, as showed up in the case of the action by conjugation, because bilinear forms are cogredient if they differ by a constant multiple. 

\section{Lie algebras arising from extensions of a trivial Lie algebra by a trivial Lie algebra}

Suppose that we are given a space $\C^{n+1}=\langle e_1\mcom e_{n+1}\rangle$. We are going to analyze how extensions of the trivial Lie algebra
structure on $W=\langle e_{n+1}\rangle$ by the trivial Lie algebra structure on $M=\langle e_1\mcom e_n\rangle$ look like. The Lie algebras are determined by
the rules
\begin{equation*}
  [e_j,e_{n+1}]=a^i_je_i.
\end{equation*}
Any matrix $A=(a^i_j)$ determines a Lie algebra. Moreover
$A$ and $A'$ determine isomorphic Lie algebras precisely when $A'=cG\inv A G$ for some block diagonal matrix  $G=\left[\begin{smallmatrix}G'&0\\0&c\end{smallmatrix}\right]$
in $\GL(n+1,\C)$.  This means that these algebras are classified by the scalar similarity classes of matrices.
We give a decomposition for some small dimensional cases, followed by a description of the general case.

\subsection{2-dimensional Lie algebras}

Every 2-dimensional Lie algebra is solvable, and so arises as an extension of a 1-dimensional Lie algebra by a 1-dimensional Lie algebra.  The matrix $A$ has the form
$A(p)=[p]$, where the coordinate $p$ is projective, in the sense that $p'=cp$ determines the same structure
when $c\ne 0$.  Of course, the case $p=0$ cannot be excluded, as it gives the trivial 2-dimensional Lie algebra. So we need to include, what algebraic geometers
call, the generic point as an element of $\P^0$.  This is true in general. In all of our constructions, we include the generic element $(0:\cdots :0)$ in $\P^n$, because it corresponds to an actual algebra. Moreover, deformations of this point can be understood easily from this point of view. The algebra corresponding to the generic point always has, what are called, \emph{jump deformations} to every other point in the same stratum. (We explain the terminology later).

\subsection{3-dimensional case}

Here we encounter the first case where we discovered a different classification than the classical one, which arises from our deformation theory point of view.
The $2\times 2$ matrices which arise, fall into two distinct strata

\begin{table}[h]
\begin{tabular}{cc}
 $ A(p_1)=\left[\begin{array}{cc}p_1&0\\0&p_1\end{array}\right]$&$B(p_1:p_2)=\left[\begin{array}{cc}p_1&1\\0&p_2\end{array}\right]$\\\\
 $\P^0$&$P^1/\Sigma_2$\\\\
 $[2,0]$&$[1,1]$
\end{tabular}
\end{table}
In the table above, the second row corresponds to the orbifold that parameterizes the stratum, and the third row gives the partition of $2$ which the stratum
corresponds to.

The second stratum has a symmetry, given by the interchange of $p_1$ and $p_2$, in other words,
$B(p_1:p_2)\sim B(p_2:p_1)$, so that this stratum is parameterized by $\P^1/\Sigma_2$. Moreover, this is the only symmetry in this stratum.
The two strata correspond to the partitions of $2$, with $A(p_1)$ corresponding to the partition $[2,0]$, and $B(p_1:p_2)$ corresponding to the
partition $[1,1]$. Moreover, the deformations of these algebras into algebras
 to other algebras of this type can be read from the corresponding partition.

First, we note that there are \emph{smooth deformations} along the family $B(p_1:p_2)$ of tshe form $B(p_1+t:p_2)$ (when $p_2\ne 0$). By a smooth deformation
we mean a deformation with a parameter $t$, such that for $s\ne t$, the element corresponding to $s$ is not equivalent to the element corresponding to $t$, at least in some small nbd of 0.  We say that the algebra deforms smoothly
in a nbd of itself along the family. When $p_2=0$, we can use a different deformation $B(p_1:t)$.  Both of these deformations are versal in terms of deforming in this class of algebras.  Finally, deformations of $B(0:0)$ along the family can be obtained by a 2-parameter family $B(t_1:t_2)$. These deformations determine \emph{jump deformations} from $B(0:0)$ to any element $B(p:q)$ along the 1-parameter subfamily $t_1=pt$, $t_2=qt$.  The term jump deformation refers
to the fact that $B(pt:qt)\sim B(p:q)$ whenever $t\ne 0$, so this means that along this curve, except at the starting point, we obtain the same equivalence
class.

Finally, we note that if $p\ne 0$, then $A_{t_1,t_2}=\left[\begin{smallmatrix}p+t_1&t_2\\0&p\end{smallmatrix}\right]$ gives rise to all of the deformations of $A(p)$. There is a type of deformation, called a \emph{versal deformation}, which has the property that it induces all deformations. In particular, there is
a special kind of versal deformation, called a \emph{miniversal deformation}, that has the fewest possible parameters in a versal deformation.
The deformation $A_{t_1,t_2}$ is not versal,  because it does not give all the deformations, even though it determines all the algebras to which $A(p)$ deforms.
From our point of view,  versal deformation is only important because it determines what algebras the original object deforms to, and we are not really interested in the versal deformation itself.

In fact, when $p\ne 0$, a miniversal deformation of $A(p)$ is given by $\left[\begin{smallmatrix}p&t_1\\t_2&p+t_3\end{smallmatrix}\right]$.
Nevertheless, the deformation $A_{t_1,t_2}$ determines all of the algebras that $A(p)$ deforms to. In fact, we have a jump deformation
$A(p)\leadsto B(p:p)$ given by the matrix $\left[\begin{smallmatrix}p&t\\0&p\end{smallmatrix}\right]$, as well as deformations in a nbd of $B(p:p)$.

 It is important to realize  that we can read off the deformations by comparing the partitions.  Consider $[2,0]$ and $[1,1]$. It is possible to obtain the first partition by moving the 1 in the second spot in the second partition to the first spot.  This means that the algebra associated to the first partition has a jump deformation to an algebra given by the second partition.  In fact, what it indicates is that if $p_2=p_1$ in $B(p_1:p_2)$, then the jump is between
 $A(p_1)$ and $B(p_1:p_2)$. This information is contained in the fact that to obtain the partition $[2,0]$ from $[1,1]$, we add the second column to the first one, corresponding to replacing the second variable with the first one.

The deformations of the algebras are captured in the picture below.

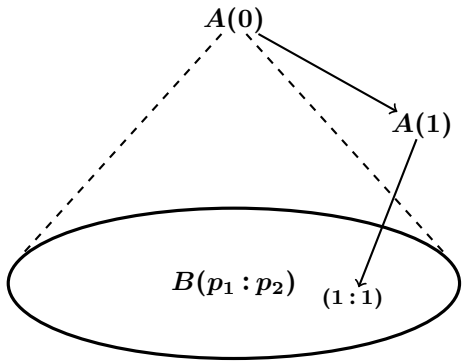
\begin{figure}[ht]
    \begin{tikzpicture}[]
    \node[algebra] at (3,4.5)    (A0)   {\bm{$A\of{0}$}};
    \node[algebra] at (5.5,3.1)  (A1)   {\bm{$A\of{1}$}};
    \node[algebra] at (3,1)      (Bpq)  {\bm{$B\of{p_1:p_2}$}};
    \node[algebra] at (4.6,0.8)  (B11)  {\footnotesize\bm{$\of{1:1}$}};
    \draw[family]         (Bpq)  ellipse (3cm and 1cm);
    \draw[jumpFamily]     (A0) -- (0.2,1.4);
    \draw[jumpFamily]     (A0) -- (5.8,1.4);
    \draw[jumpSingle]     (A0) -- (A1);
    \draw[jumpSingle]     (A1) -- (B11);
    \end{tikzpicture}

    \caption{The Moduli Space of Similar $2\times 2$ Complex Matrices}

\end{figure}

Now, there is a bit more complexity in the space of 3-dimensional algebras, because they are not all given by extensions of a trivial algebra by another one.
In fact, there is a 3-dimensional simple algebra $\sl(2,\C)$. In terms of deformations of 3-dimensional Lie algebras, the algebra corresponding to
$B(1:-1)$ has an extra jump deformation to the algebra $\sl(2,\C)$, so the actual deformation picture is a bit more complex than we are representing. What we
are discussing here is only the part of the moduli space determined by these $2\times 2$ matrices.

Now, let us consider the classical decomposition of the moduli space of  3-dimensional Lie algebras (see \cite{jac}). There are several types coming from the 2-dimensional matrices above, with only the simple Lie algebra not corresponding to this kind of decomposition.
\begin{align*}
\begin{array}{cccc}
\mathfrak r_{3,\lambda}(\C)&\mathfrak r_3(\C)&\mathfrak r_2(\C)\oplus\C&\mathfrak n_3\\\\
\left[\begin{array}{cc}1&0\\0&\lambda\end{array}\right]&
\left[\begin{array}{cc}1&1\\0&1\end{array}\right]&
\left[\begin{array}{cc}1&1\\0&0\end{array}\right]&
\left[\begin{array}{cc}0&1\\0&0\end{array}\right]
\end{array}
\end{align*}
Note that the main difference between the classical decomposition and ours is that the elements $\left[\begin{smallmatrix}1&0\\0&1\end{smallmatrix}\right]$
and $\left[\begin{smallmatrix}1&1\\0&1\end{smallmatrix}\right]$ are interchanged. The question is which one really should belong to the family? The answer
is that the second one is the correct one, because the first one has a jump deformation to the second one, while the second one behaves in the generic manner
of the other elements in the family.  This fact is revealed in the cohomology of the corresponding algebras, but is not important for our present discussion.
Note that our description is more compact, but that is not the justification for our decomposition.  The main advantage is that the description in terms of the projective orbifold structure captures the geometric picture of the moduli space.

\section{4-dimensional Lie algebras}
Here, we are studying the equivalence classes of $3\times 3$ matrices. The deformation information is represented in the following picture.

\begin{figure}[h]
    \begin{tikzpicture}[]
    \node[algebra] at (3,7)      (A0)    {\bm{$A\of{0}$}};
    \node[algebra] at (11,5)     (A1)    {\bm{$A\of{1}$}};
    \node[algebra] at (9,3)      (Bpq)   {\bm{$B\of{p_1:p_2}$}};
    \node[algebra] at (10.2,2.8) (B11)   {\footnotesize\bm{$\of{1:1}$}};
    \node[algebra] at (3,1)      (Cpqr)  {\bm{$C\of{p_1:p_2:p_3}$}};
    \node[algebra] at (4.6,0.6)  (Cppq)  {\footnotesize\bm{$\of{p_1:p_1:p_2}$}};
    \draw[family]         (Cpqr)  ellipse (3cm and 1cm);
    \draw[family]         (Bpq)   ellipse (2cm and 0.66cm);
    \draw[jumpFamily]     (A0) -- (0.2,1.4);
    \draw[jumpFamily]     (A0) -- (5.8,1.4);
    \draw[jumpFamily]     (A0) -- (7,3);
    \draw[jumpFamily]     (A0) -- (10.6,3.4);
    \draw[jumpFamily]     (7,3) -- (Cppq);
    \draw[jumpFamily]     (10.4,2.5) -- (Cppq);
    \draw[jumpSingle]     (A0) -- (A1);
    \draw[jumpSingle]     (A1) -- (B11);
    \end{tikzpicture}

    \caption{The Moduli Space of Similar $3\times3$ Complex Matrices}

\end{figure}
\begin{table}[h]
\begin{tabular}{ccc}
  $A(p)=\left[\begin{array}{ccc}p&0&0\\0&p&0\\0&0&p\end{array}\right]$&$B(p_1:p_2)=\left[\begin{array}{ccc}p_1&0&0\\0&p_1&1\\0&0&p_2\end{array}\right]$&
  $C(p_1:p_2:p_3)=\left[\begin{array}{ccc}p_1&1&0\\0&p_2&1\\0&0&p_3\end{array}\right]$\\\\
  $\P^0$&$\P^1$&$\P^2/\Sigma_3$\\
  $[3,0,0]$&$[2,1,0]$&$[1,1,1]$\\
\end{tabular}
\end{table}

Note that we can read off the deformations directly from the partition. For example $[3,0,0]$ can be obtained from $[2,1,0]$ by adding the second column to the first, so there is a corresponding deformation.  The deformation picture does not require computing a versal deformation.

\section{5-dimensional Lie algebras}

The table below gives the stratification of the space of $4\times 4$ matrices by projective orbifolds.

\begin{table}[h]
  \begin{tabular}
    {ccccc}
    $A(p_1)$&$B(p_1:p_2)$&$C(p_1:p_2)$&$D(p_1:p_2:p_3)$&$E(p_1:p_2:p_3:p_4)$\\\\
    $\left[\begin{smallmatrix}p_1&0&0&0\\0&p_1&0&0\\0&0&p_1&0\\0&0&0&p_1\end{smallmatrix}\right]$&
    $\left[\begin{smallmatrix}p_1&0&0&0\\0&p_1&0&0\\0&0&p_1&1\\0&0&0&p_2\end{smallmatrix}\right]$&
    $\left[\begin{smallmatrix}p_1&1&0&0\\0&p_2&0&0\\0&0&p_1&1\\0&0&0&p_2\end{smallmatrix}\right]$&
    $\left[\begin{smallmatrix}p_1&0&0&0\\0&p_1&1&0\\0&0&p_2&1\\0&0&0&p_3\end{smallmatrix}\right]$&
     $\left[\begin{smallmatrix}p_1&1&0&0\\0&p_2&1&0\\0&0&p_3&1\\0&0&0&p_4\end{smallmatrix}\right]$\\\\
     $\P^0$&$\P^1$&$\P^1/\Sigma_2$&$\P^2/\Sigma_2$&$\P^3/\Sigma_4$\\\\
     $[4,0,0,0]$&$[3,1,0,0]$&$[2,2,0,0]$&$[2,1,1,0]$&$[1,1,1,1]$
  \end{tabular}
\end{table}

The picture corresponding to this stratification is given below.

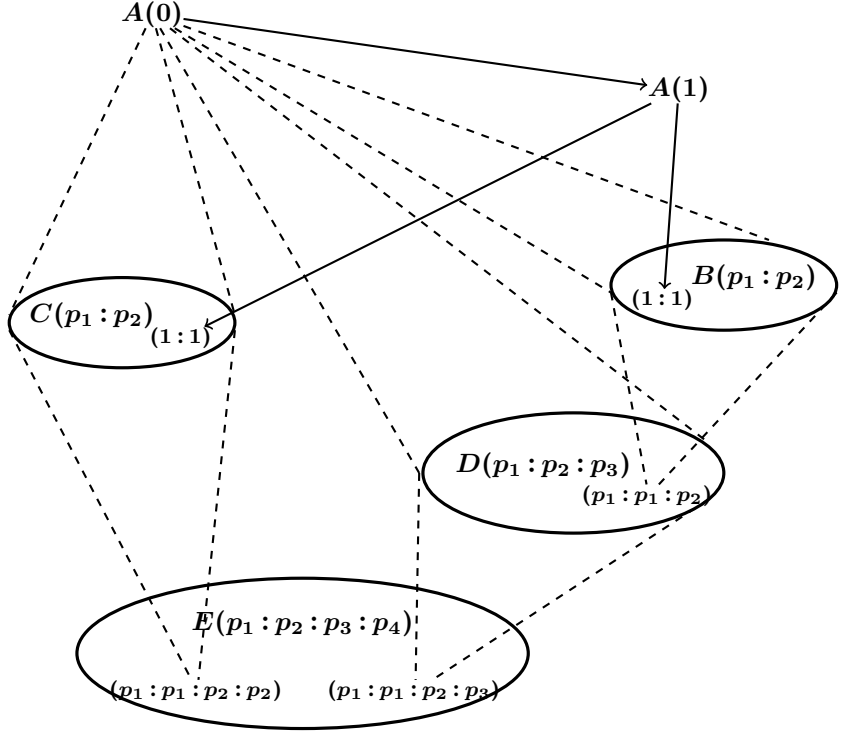
\begin{figure}[h]
    \begin{tikzpicture}[]
    \node[algebra] at (2,9.5)      (A0)    {\bm{$A\of{0}$}};
    \node[algebra] at (9,8.5)     (A1)    {\bm{$A\of{1}$}};
    \node[algebra] at (10,6)      (Bpq)   {\bm{$B\of{p_1:p_2}$}};
    \node[algebra] at (8.8,5.7) (B11)   {\footnotesize\bm{$\of{1:1}$}};
    \node[algebra] at (1.2,5.5)   (Cpq)  {\bm{$C\of{p_1:p_2}$}};
    \node[algebra] at (2.4,5.2) (C11)   {\footnotesize\bm{$\of{1:1}$}};
    \node[algebra] at (7.2,3.5)      (Dpqr)  {\bm{$D\of{p_1:p_2:p_3}$}};
    \node[algebra] at (8.6,3.1) (Dppq)   {\footnotesize\bm{$\of{p_1:p_1:p_2}$}};
    \node[algebra] at (4,1.4)      (Epqrs)  {\bm{$E\of{p_1:p_2:p_3:p_4}$}};
    \node[algebra] at (5.5,0.5)  (Eppqr)  {\footnotesize\bm{$\of{p_1:p_1:p_2:p_3}$}};
    \node[algebra] at (2.6,0.5)  (Eppqq)  {\footnotesize\bm{$\of{p_1:p_1:p_2:p_2}$}};
    \draw[family]         (Bpq)+(-0.4,-0.1)   ellipse (1.5cm and 0.6cm);
    \draw[family]         (Cpq)+(0.4,-0.1)   ellipse (1.5cm and 0.6cm);
    \draw[family]         (Dpqr)+(0.4,-0.1)  ellipse (2cm and .8cm);
    \draw[family]         (Epqrs)+(0,-0.4)  ellipse (3cm and 1cm);
    \draw[jumpFamily]     (A0) -- (0.1,5.5);
    \draw[jumpFamily]     (A0) -- (3.1,5.5);
    \draw[jumpFamily]     (A0) -- (5.55,3.4);
    \draw[jumpFamily]     (A0) -- (9.4,3.8);
    \draw[jumpFamily]     (A0) -- (8.1,5.8);
    \draw[jumpFamily]     (A0) -- (10.2,6.5);
    \draw[jumpFamily]     (5.55,3.4) -- (Eppqr);
    \draw[jumpFamily]     (9.4,3.05) -- (Eppqr);
    \draw[jumpFamily]     (0.1,5.3) -- (Eppqq);
    \draw[jumpFamily]     (3.1,5.3) -- (Eppqq);
    \draw[jumpFamily]     (8.1,5.8) -- (Dppq);
    \draw[jumpFamily]     (11.1,5.8) -- (Dppq);
    \draw[jumpSingle]     (A1) -- (C11);
    \draw[jumpSingle]     (A1) -- (B11);
    \draw[jumpSingle]     (A0) -- (A1);
    \end{tikzpicture}

    \caption{The Moduli Space of Similar $4\times 4$ Complex Matrices}
\end{figure}
Once again, the deformations of the strata can be determined by the partition.  We also can read off the orbifold structure of the stratum from the matrix.
For example, in the fourth stratum, one can interchange $p_2$ and $p_3$, but not the other parameters, because the repetition of $p_1$ makes this parameter play
a different kind of role.

We also want to emphasize that the deformations given by these matrices don't give the complete deformation picture of the corresponding Lie algebras, because
there are other strata which don't arise by extensions of a 1-dimensional Lie algebra by a 4-dimensional Lie algebra, so the deformation picture here is much
simpler to describe than the general picture.  For example, in this context, every element except the generic element has the same number of parameters in a miniversal deformation, but in reality, there are special subfamilies with different deformations in the Lie algebra picture. That is because there are
other elements in the moduli space which are not given by those $4\times 4$ matrices, so there are more directions in which an element may deform.

\section{General Case}

From the examples above, it is fairly straightforward to write the general picture of the moduli space of $n\times n$ matrices under the action of
$\GL(n,\C)\times\C^*$.  Consider a multi-index of the form $[m_1\mcom m_n]$, where $m_1\mplus m_n=n$, $m_i\ge m_{i+1}\ge 0$. Then each such multi-index
determines a stratum in the following manner. First, we consider the case $m_1>m_2$. Then the matrix should have $m_1-m_2$ columns with only a $p_1$ on the main diagonal.
For the $(m_1-m_2+1)$-th row, there will be a 1 to the right of the entry $p_1$. Next, suppose that $m_2=m_3=\cdots m_k>m_{k+1}$. Then there will be $m_k-m_{k+1}$ repetitions of the pattern where the columns have a 1 above the entry on the main diagonal except in the first column, followed by the entries on the main diagonal given by $p_1,\mcom p_k$ sequentially, repeating this pattern $m_k-m_{k+1}$ times. If we think that this has reduced all the entries in the multi-index to $m_{k+1}$, then
we repeat the process again,  until finally we have run out of nonzero entries in the multi-index.  If the first $N$ entries in the multi-index are nonzero, then
there will be variables $p_1\mcom p_N$, and the stratum will be parameterized by $\P^{N-1}$ with the action of a subgroup of $\Sigma_{N}$.  Deformations of 
the elements in the stratum can be read off easily.

We give an example of a stratum in the space of $10\times 10$ matrices given by the multi-index $[3,2,2,2,1,0,0,0,0]$. The resulting matrix will have 5 parameters
$p_1\mcom p_5$, and will be given by equivalence classes of matrices of the form
$$
\left[
\begin{array}{rrrrrrrrrr}
p_1&0&0&0&0&0&0&0&0&0\\
0&p_1&1&0&0&0&0&0&0&0\\
0&0&p_2&1&0&0&0&0&0&0\\
0&0&0&p_3&1&0&0&0&0&0\\
0&0&0&0&p_4&0&0&0&0&0\\
0&0&0&0&0&p_1&1&0&0&0\\
0&0&0&0&0&0&p_2&1&0&0\\
0&0&0&0&0&0&0&p_3&1&0\\
0&0&0&0&0&0&0&0&p_4&1\\
0&0&0&0&0&0&0&0&0&p_5\\
\end{array}
\right]
$$
The stratum is parameterized by $\P^4/\Sigma_3$, where $\Sigma_3$ acts by permuting the coordinates $p_2$, $p_3$ and $p_4$.

Every nonincreasing multi-index $m=[m_1\mcom m_n]$ of nonnegative integers determines a stratum in the moduli space of $n\times n$ complex matrices.  When do elements in
the stratum $m$ deform to elements in the stratum $k=[k_1\mcom k_n]$?  It turns out that if we can get from $k$ to $m$ by adding some columns in $k$ to each other and rearranging, then there is a deformation from $m$ to $k$.

Finally, we note that the number of strata in the moduli space of $n\times n$ matrices is exactly equal to the number of partitions of $n$.

\section{Algebraic definition of a formal deformation}

If $d$ is a Lie algebra structure on a vector space $V$ over a field $\k$, then a 1-parameter deformation of $d$ is an algebra of the form
\begin{align*}
  d_t=d+t\psi_1+t^2\psi_2\mplus
\end{align*}
where $\psi_i:V\tns V\ra V$ are antisymmetric functions.  Of course, we require that $d_t$ is a Lie algebra, which is a condition related to the
Schouten bracket of the map, but we don't need to go into the details of this construction here (see \cite{NR}).

One can define multiparameter deformations as well, but there is a technical definition of a formal deformation as follows.  A \emph{formal base} $A$ over
$\k$ is a complete, local algebra over $\k$. In fact, a formal base is nothing more than a quotient of a formal power series algebra $\k[[t_1\mcom]]$ by
an ideal, so that it can be given in terms of parameters $t_k$. A deformation $d_A$ of $d$ with base $A$ is an $A$-Lie algebra structure on $V\tens A$, which
projects to $d$ under the canonical morphism $A\ra\k$.  A morphism of algebras $f:A\ra B$ induces a deformation $f_*(d_A)$ with base $B$, see \cite{fi}.

A versal deformation $d_A$ with base $A$ is one such that if $d_B$ is any formal deformation with formal base $B$, then there is a morphism $f:A\ra B$ such that
$f_*(d_A)$ is equivalent to $d_B$.  The reason that it is called versal, rather than universal is that the morphism $f$ is not unique, in general.  A
miniversal deformation of $d$ is one for which the morphism $f$ is unique when the formal base $B$ is \emph{infinitesimal}, meaning that the square of the maximal ideal in $B$ vanishes.  What this means in practice is that the number of parameters in the base $A$ is minimal.

This technical definition is abstract, but in practice, the computation of a miniversal deformation is very concrete, see \cite{fifu}.

\section{Analytic definition of Miniversal Deformations}

We give a more formal definition of a deformation of a group action, due to Arnold \cite{arn}. We restrict ourselves to the case when we are working with complex vector spaces.  It is also enough to consider matrix groups, because a group action on a finite dimensional
vector space reduces to the action of a matrix group.
\begin{dfn}
Let $G$ be a matrix group acting on a complex vector space $V$. Then a deformation of an element $\bar v\in V$ is a holomorphic map $v(t):\C^k\ra V$, defined in some
nbd of zero, such that
$\overline{v_0}=\bar v$. Two deformations $v(t):\C^k\ra V$ and $\hat v(t):\C^k\ra V$ are called equivalent if there is a deformation $G(t):\C^k\ra G$ of the identity matrix $I\in G$, such that
$G(t)v(t)=\hat v(t)$.
\end{dfn}
Notice that in this definition, it is not sufficient that for each $t$ there is a matrix $G(t)$ such that $G(t)v(t)=\hat v(t)$. The dependence of $G(t)$ on $t$ needs
to be holomorphic, and moreover, $G(0)$ must be the identity matrix. This corresponds to the notion of \emph{formal equivalence} in the deformation theory of algebras.
Also, we should mention that by a deformation of the identity matrix we mean a holomorphic map $G(t):C^k\ra G$, defined in a nbd of zero, such that
$G(0)=I$.

\begin{dfn}
A deformation $v(t):\C^k\ra V$ of $\bar v$ is called \emph{versal}, if given any other deformation $u(s):\C^\ell\ra V$ of $v$, there is a holomorphic function
$\ph:\C^\ell\ra C^k$, defined in a nbd of $0$, such that $\hat v(s)=v(\ph(s))$.
\end{dfn}
For a deformation $V(t):\C^k\ra V$, the dimension $k$ is called the number of parameters of the deformation.
\begin{dfn}
A deformation $v(t)$ of $\bar v$ is said to be a \emph{miniversal} deformation of $v$ if it is versal and the number of parameters is the minimum among all
versal deformations of $\bar v$.
\end{dfn}

At first, this definition seems far away from the definition of a deformation of an algebra, where we consider deformations given over a \emph{formal base}
which is a complete, local algebra $A$.  The local property is that there is a unique maximal ideal $\mathfrak m$ in $A$, while the completeness means that
expressions of the form $\sum_{i=0}^\infty a_i$ are well defined, as long as $a_i$ lies in the $i$th power of the ideal.  However, there is a base for
a deformation in the above definition, just a little bit more hidden.  The base is the germs of analytic functions on $\C^k$, and the maximal ideal is
given by the functions which vanish at the origin. The completeness follows from the fact that the functions are represented by power series converging
in a nbd of the origin.  In this sense, the algebraic definition is more general, but for the application in mind, the analytic definition is more natural.

If one considers $Gv$, the orbit of $v$ under the group action, it is not generally a manifold, so the notion of its tangent space is somewhat problematic.  However,
suppose we consider a 1-parameter deformation $g(t)$ of the identity matrix, we can compute $(d(g(t)v)/dt)_{t=0}$, which is a tangent vector to $v$.  The subspace $T$ of $V$ spanned
by the tangent vectors to $v$ is called the tangent space of $v$.  The following theorems make it possible to compute miniversal deformations (see, for example
\cite{arn}).
\begin{thm}
Let $V=\langle v_1\mcom v_n\rangle$ be a finite dimensional vector space with a group action.  Define $v(t)=v+\sum_{i=1}^n t_iv_i$. Then
$v(t)$ is a versal deformation of $\bar v$.
\end{thm}
We also have a characterization of a miniversal deformation of $\bar v$.
\begin{thm}
Let $T$ be the tangent space to $v$ and suppose that $W=\langle w_1\mcom w_k\rangle$ is a complementary subspace of $V$ to $T$. Then
$v(t)=v+\sum_{i=1}^k t_iw_i$ is a miniversal deformation of $\bar v$.
\end{thm}

\section{Arnold's Decomposition of the moduli space of Matrices under similarity}

One of the nice results in \cite{arn} is the determination of the number of parameters of a miniversal deformation of a matrix $A_0$ in terms of the
Jordan decomposition of the matrix.
\begin{thm}[Arnold] If $A_0$ is a matrix, then the number $n$ of parameters of a miniversal deformation of $A_0$ is given by
\begin{equation*}
n=\sum_\lambda n_1+3n_2+5n_3+\cdots,
\label{arneq}
\end{equation*}
where the sum is taken over all eigenvalues $\lambda$ of the matrix, and $n_1\ge n_2\ge\cdots$ are the sizes of the Jordan blocks corresponding to $\lambda$.
\end{thm}
 The basic idea of the proof is as follows. If we consider a one-parameter subgroup
$\exp(tB)$ of $\GL(n)$, then a tangent vector to the action of this subgroup is $\frac d{dt}\,\exp(tB)^*(A)=[A,B]$. Consider the map $f$ defined on the $n\times n$ matrices given by $f(B)=[A,B]$. Then the dimension of the tangent space to the group action at $A$ is $\dim(\im(f))$, and so its codimension, $\dim(\ker(f))$, is just the dimension of the centralizer of $A$.  This means the dimension of the miniversal deformation is given by the dimension of the centralizer of $A$. Arnold goes on to compute this dimension by looking at a Jordan normal form for $A$, and explicitly constructing its centralizer, from which he concludes the dimension formula above.

One should note that the centralizer subspace is not, in general, transverse to the tangent space of the group action, so that one cannot use the centralizer to compute a miniversal deformation directly. Nevertheless, Arnold gives an explicit form for a miniversal deformation based on the Jordan normal form of the matrix.

The decomposition given by Arnold is in terms of the Jordan decomposition of a matrix. It has been pointed out by others \cite{dfs} that this decomposition is
problematic because it doesn't fit well with deformation theory.  When we first encountered Arnold's theory, we noticed that for every matrix in a stratum
in our decomposition, the number of parameters in the miniversal deformation was the same, regardless of its Jordan decomposition,
and as given by Arnold, with the exception that his count of the parameters was one higher than ours
with the exception of the generic element, when the number of parameters coincided. This makes sense, as there is one less direction to deform except for the generic
matrix.

For example, consider the stratum $B(p:q:r)$ in the 4-dimensional algebras.  First, consider generic values for $p$, $q$ and $r$ (in other words,
they don't coincide). Then for $p$ we obtain $n_1=1$ and $n_2=1$, so $n_1+3n_2=4$, and the other variables each contribute 1, so the total number of parameters
is 6.  On the other hand, suppose $p=q\ne r$. Then $n_1=2$ and $n_2=1$, so we obtain 5 parameters, and $r$ contributes 1, so again we obtain 6. The reader
can easily check that the remaining cases also result in 6 parameters.

This completes the description of the moduli space of complex $4\times 4$ matrices under the action of the group $\GL(4,\C)\times\C^*$ by conjugation.

\subsection{The Moduli space of $n\times n$ matrices under the action of $\GL(n)$ by conjugation}
If we consider $n\times n$ matrices under the action of $\GL(n,\C)$ by conjugation, then the picture is similar to the case we have
studied above. The strata are no longer parameterized by projective coordinates, but by $\C^k$ modulo a subgroup of $\Sigma_k$. We still have
the decomposition by multi-indices as discussed above, and the deformation picture is similar.

What happens when we consider other fields than $\C$?  Of course, for an algebraically closed field, the theory of Jordan decomposition is the same,
so we obtain a similar stratification.  For non algebraically closed fields, we need to use a more complicated process.  In the study of Lie algebras
over the real numbers, we worked out a stratification for some low dimensional moduli spaces of Lie algebras, and discovered that for the action
of $\GL(n,\R)\times \R^*$ on $n\times n$ real matrices, there was a stratification by orbifolds based on $S^k$ modulo a group, spheres rather than
projective spaces.  We do not go into details here.

\section{Moduli spaces of complex bilinear forms on an $n$-dimensional space}
In the study of moduli spaces of algebras we ran into some strata which were given by the action of $\GL(n,\C)$ on $n\times n$ matrices, given by
$G\star A=G^TAG$. This is of course, the equivalence classes of bilinear forms on $\C^n$, or the equivalence classes of matrices under the cogredient action. 
This type of action arises when considering a certain type of 1-dimensional extension of an associative algebra. If $M=\langle e_1\rangle$ and $W=\langle e_2\mcom e_{n+1}\rangle$, then the extension is determined by a set of brackets of the form
\begin{align*}
  e_ie_j=a_{ij}e_1.
\end{align*}
When the algebra structure on $W$ is nontrivial, there are some conditions on the matrix $A=(a_{ij})$. In addition, the elements $G$ which act on $A$ are restricted by the requirement that $G$ must preserve the multiplication structure on $W$, so we obtain a subgroup of $\GL(n,\C)$ acting on a subspace of bilinear forms.  However, when the algebra structure on $W$ is trivial, we obtain precisely the cogredient action above, so the algebras are determined by the equivalence classes of matrices under the cogredient action.

A lot of the literature assumes that, because it is easy
to classify symmetric and antisymmetric forms, and every matrix is uniquely a sum of a symmetric and antisymmetric matrix, this means that the
classification of bilinear forms is simple.  Of course, the problem with this approach is that if a matrix $C$ decomposes as the sum of a symmetric matrix $S$ and
an antisymmetric matrix $A$, so $C=S+A$, then $G\star C=G^TSG+G^TAG$, so that while it is easy to put one of the matrices $S$ or $A$ into a canonical form, there is
no reason to assume that one can also put the other matrix in a canonical form.

Luckily, some other authors have correctly addressed these issues, and a decomposition of matrices into certain strata under the cogredient action was obtained in \cite{horn-ser}.
The idea was refined in \cite{dfs}. The decomposition given there is as follows.  Let $J_n(\lambda)=\left[\begin{array}{cccc}\lambda&1&&0\\&\lambda&\ddots&\\&&\ddots&1\\0&&&\lambda\end{array}\right]$ be the Jordan block of size $n$ with eigenvalue $\lambda$, and $\Gamma_n$ be the matrix  $\Gamma_n=\left[\begin{array}{ccccc}0&&&&\udots\\&&&-1&\udots\\&&1&1&\\&-1&-1&&\\1&1&&&0\end{array}\right]$.
Then, we have the following theorem, which was originally obtained in \cite{hor-ser2}.
\begin{thm}[Horn, Sergeichuk]\label{horn-ser}
Any square matrix is congruent to a direct sum of matrices of the form
$$
H_m(\lambda)=\left[\begin{array}{cc}0&I_m\\J_m(\lambda)&0\end{array}\right],\qquad \Gamma_m,\qquad J_m(0),
$$
where $\lambda\ne 0$, $\lambda\ne\s{m+1}$, $I_m$ is the $m\times m$ identity matrix, $H_m(\lambda)$ is equivalent to $H_m(\lambda^{-1})$, and this decomposition is unique
up to reordering of the summands.
\end{thm}

This important theorem does not give a decomposition of the moduli space into strata given by projective orbifolds. Since we had obtained a decomposition of these moduli spaces into projective strata for some small dimensional spaces, we became curious about the relation between our decomposition into projective strata and the stratification
given by the theorem above.

First, we note that our moduli space is projective in the sense that $A\sim uA$ for any $u\in\C^*$. This is easy to see because if we choose $a$ such that
$a^2=u$, then $uA=g*A$ where $g=aI_n$. Next, we study what happens for some small values of $n$. A discussion of this appeared in \cite{fp17}.

For $n=1$, there are only two equivalence classes $[1]$ and $[0]$, which clearly give a decomposition of the $1\times 1$ matrices as a $\P^0$.

Next, for $n=2$, let us consider the decomposition given by the theorem above.
The matrices which appear in the theorem are $H_1(\lambda)=\left[\begin{array}{cc}0&1\\\lambda&0\end{array}\right]$, $\Gamma_2=\left[\begin{array}{cc}0&-1\\1&1\end{array}\right]$ and
$J_2(0)=\left[\begin{array}{cc}0&1\\0&0\end{array}\right]$. We also have to consider the direct sum decompositions given by the diagonal
matrices $\diag(1,1)$, $ \diag(1,0)$ and
$\diag(0,0)$. According to the classification rules, in $H_1(\lambda)$ we must exclude $\lambda=0$ and $\lambda=1$, although the matrix $H_1(0)$ appears since it coincides with $J_2(0)$.

Now, consider the matrices

$$A(p)=\left[\begin{array}{cc}0&p\\-p&0\end{array}\right],\qquad B(p:q)=\left[\begin{array}{cc}1&p\\q&0\end{array}\right],\qquad
.$$

First we establish, that we have covered all possible equivalence classes by our classification.

\begin{thm} Every complex bilinear form can be represented by a matrix of type $A(p)$ or $B(p:q)$.
\end{thm}
\begin{proof}
If $\beta$ is a bilinear form on $\C^2$ which is antisymmetric, then clearly it is representable by a matrix of type $A(p)$.  Suppose that
$\beta$ is a bilinear form, which is not antisymmetric.  Since $\beta$ cannot be zero, there is some $u\in\C^2$ such that $\beta(u,u)\ne 0$.
Choose any $v\in\C^2$ so that $\C^2=\langle u,v\rangle$. 
Let $c$ in $\C^*$ be such that $c^{-2}=\beta(u,u)$, and let $e_1=cu$. then $\beta(e_1,e_1)=1$. Now if $\beta(v,v)=0$, then let $e_2=v$. Otherwise
let $e_2=e_1+xv$, and then we compute that $\beta(e_2,e_2)=1+x(\beta(e_1,v)+\beta(v,e_1))+x^2\beta(v,v)$. So there is an $x$ which solves this
quadratic equation, yielding a matrix of type $B(p:q)$.
\end{proof}

Next, we classify the matrices of the form $A(p)$. Actually, it is easy to see that $A(p)$ is equivalent to $A(cp)$ when $c\ne 0$, so $A(p)$ gives
a stratum parametrized by $\P^0$.

The stratum $B(p:q)$ is parametrized by $\P^1/\Sigma_2$, where $\Sigma_2$ acts in the usual manner by interchanging coordinates.
\begin{thm}
$B(p:q)$ is equivalent to $B(cp:cq)$ for $c\in\C^*$, so the stratum $B(p:q)$ has projective coordinates.  Moreover
$B(p:q)$ is equivalent to $B(x:y)$ iff $(p:q)=(x:y)$ or $(p:q)=(y:x)$, which means the the stratum is parametrized by $\P^1/\Sigma_2$.
\end{thm}
\begin{proof}
To show the first statement, we find a matrix $G$ such that $$G^TB(p:q)G=B(cp:cq.$$ For this purpose it is sufficient to take $G=\diag(1,c)$.
To show the second statement, the matrix $G=\left[\begin{array}{cc}1&-p-q\\0&1\end{array}\right]$ transforms $B(p:q)$ into $B(q:p)$. It is not hard
to see that this covers all possibilities.
\end{proof}
Thus we see that the moduli space of $2\times 2$ complex matrices is stratified by projective orbifolds of the type we have been discussing throughout the paper.

Let us compare the stratification above to the decomposition in Theorem (\ref{horn-ser}).

\begin{thm}The following dictionary between the two classification schemes holds.
\begin{enumerate}
\item If $\lambda=p/q$ is not 1, -1, or 0, then $H_1(\lambda)\sim B(p:q)$.
\item $B(1:1)\sim\diag(1,1)$.
\item $B(1:-1)\sim\Gamma_2$.
\item $B(1:0)\sim J_2(0)$.
\item $B(0:0)\sim\diag(1,0)$.
\item $H_1(-1)\sim A(1)$.
\item $A(0)=\diag(0,0)$.
\end{enumerate}
\end{thm}
It is easy to construct matrices which carry out these equivalences.  Let us explain why we have now completely identified the two classifications.  Note that
replacing $\lambda$ by $\lambda^{-1}$ in $H_1(\lambda)$ corresponds to interchanging $p$ and $q$ in $B(p:q)$.  The matrices $H_1(1)$ and $H_1(0)$ are excluded
from the first classification scheme, but not $H_1(-1)$, so we had to account for both what $B(1: -1)$ and $H_1(-1)$ are equivalent to in the opposite classifications, and we did. We have accounted for the matrices $\Gamma_2$ and $J_2(0)$, as well as the diagonal matrices which result from bilinear forms which decompose as direct sums. Thus we have accounted for everything on both sides in a unique (up to the symmetries of each type) manner.  One advantage of our approach here is that we get only
2 strata, and we can account for the geometry in a simple manner.

It is easy to see that elements in the family $B(p:q)$ can only deform in a nbd of the point within the family, so there is a 1-parameter miniversal deformation,
which can be given by
$$
B(p:q)_t=B(p:q)+te_{2,2}=\left[\begin{array}{cc}1&p\\q&t\end{array}\right].
$$
The exception to the above is the generic element $B(0:0)$ which has a miniversal deformation
$$
B(0:0)_t=B(0:0)+t_1e_{1,2}+t_2e_{2,2}=\left[\begin{array}{cc}1&t_1\\0&t_2\end{array}\right],
$$
which means it jumps to every element $B(x:y)$ except itself.  In fact, $B_t=\left[\begin{array}{cc}1&t_1\\t_2&0\end{array}\right]$ will give all the deformations of $B(0:0)$, although it is not a miniversal deformation. The reason is that the tangent space is spanned by the matrices
$T_1=\left[\begin{smallmatrix}1&0\\0&0\end{smallmatrix}\right]$ and $T_2=\left[\begin{smallmatrix}0&1\\1&0\end{smallmatrix}\right]$, so the matrices
$\left[\begin{smallmatrix}0&1\\0&0\end{smallmatrix}\right]$, and $\left[\begin{smallmatrix}0&0\\1&0\end{smallmatrix}\right]$ are not linearly independent
from $T_2$.

Even the fact that $T_2$ lies in the tangent space requires a bit of work to see.  If we let $B_t=\left[\begin{array}{cc}t_1&t_2\\t_2&\tfrac{t_2^2}{1+t_1}\end{array}\right]$, then for $G=\left[ \begin {array}{cc} {\frac {1}{\sqrt {1+t_{{1}}}}}&-{\frac {t_{{2}}}{1+t_{{1}}}}\\\noalign{\medskip}0&1\end {array} \right]$, it can be shown that $G^TB_tG=B(0:0)$. Note that $G=G(t_1,t_2)$ is analytic in a nbd of the origin in the
parameter space and $G(0,0)=I$.

Also, a small deformation of the form $A(1)+te_{1,1}$ will give a jump to $B(1: -1)$, and the miniversal deformation is given by
$$A(1)_t=A(1)+t_1e_{11}+t_2e_{12}+t_3e_{22}=\left[\begin{array}{cc}t_1&1+t_2\\-1&t_3\end{array}\right],$$
 where $t_2$ governs smooth deformations in a nbd of $B(1: -1)$.

The generic element $A(0)$ has a more complex
miniversal deformation
$$A(0)_t=A(0)+t_1e_{11}+t_2e_{12}+t_3e_{21}+t_4e_{22}=\left[\begin{array}{cc}t_1&t_2\\t_3&t_4\end{array}\right],$$
which, when $t_4=0$,  is equivalent to $B(t_2:t_3)$ unless $t_1=0$ and $t_2=-t_3$, in which case it is equivalent to $A(1)$ (unless all the parameters vanish),
which means it jumps to every element $B(x:y)$ and to $A(1)$.  Of course, this is obvious from the fact that the zero vector must deform to every vector.

We are mainly interested in the geometry of the moduli space, so our purpose in constructing a miniversal deformation is really to study what the element
deforms to, rather than the abstract purpose of finding a miniversal deformation.  It is often possible to understand what something deforms to from
something simpler than the miniversal deformation.  The geometry depends on understanding how the space is assembled from natural strata, and the jump
deformations give some type of gluing information about the space.  There is a unique stratification which captures  the geometric information, but there
can be many different classification schemes, each of which has a different purpose and flavor.

In the next section, where we study bilinear forms on a 3-dimensional complex vector space, we will not construct miniversal deformations, and will just
give a description of what the elements deform to, and whether they deform smoothly or jump.

\subsection{Bilinear forms on a 3-dimensional complex space}
Using the ideas from \cite{horn-ser}, one could easily determine a classification of the 3-dimensional bilinear forms by the following matrices:
\begin{align*}
B_1(p:q)&=\left[\begin{array}{ccc}1&p&0\\q&0&0\\0&0&1\end{array}\right],
&B_2(p:q)=\left[\begin{array}{ccc}1&p&0\\q&0&0\\0&0&0\end{array}\right],
B_3&=\left[\begin{array}{ccc}0&-1&0\\1&0&0\\0&0&1\end{array}\right],\\
B_4&=\left[\begin{array}{ccc}0&-1&0\\1&0&0\\0&0&0\end{array}\right],
&B_5=\left[\begin{array}{ccc}0&1&0\\0&0&1\\0&0&0\end{array}\right],
B_6&=\left[\begin{array}{ccc}0&0&1\\0&-1&-1\\1&1&0\end{array}\right].
\end{align*}
The matrix $B_5=J_3(0)$ and $B_6=\Gamma_3$ are the two indecomposable matrices occurring in the Horn-Sergeichuk classification, while the matrices $B_1(p:q)$,
$B_2(p:q)$, $B_3$ and $B_4$ correspond to the decomposable bilinear forms. Here we used our explicit classification of 1 and 2 dimensional complex bilinear forms.
It is easy to check that $B_1(0:0)\sim B_2(1:1)$, and other than this identification, there is no overlap.  Moreover, each of the strata given by projective coordinates
$(p:q)$ is parametrized by $\P^1/\Sigma_2$.  Thus it would seem that this decomposition satisfies all the requirements which we have given.  However, our motivation
behind the classification scheme was to apply it in an algebraic setting to determine a stratification of a moduli space of algebras which is compatible with deformations, and the stratification we just gave has some problems in this regard.  Although our classification here is of the moduli space of the action of a group on
a vector space, the deformations correspond to deformations of algebras, so we found a better decomposition.

Consider the matrices below:
\begin{align*}
C_1(p:q)&=\left[\begin{array}{ccc}0&0&q\\0&1&1\\p&0&1\end{array}\right],\quad
C_2=\left[\begin{array}{ccc}0&1&0\\-1&0&0\\0&0&1\end{array}\right],\quad
C_3=\left[\begin{array}{ccc}0&1&1\\-1&0&0\\0&0&0\end{array}\right],\\
C_4&=\left[\begin{array}{ccc}0&1&0\\1&1&0\\0&0&1\end{array}\right],\quad
C_5(p:q)=\left[\begin{array}{ccc}0&0&q\\0&0&0\\p&0&1\end{array}\right],\quad
C_6=\left[\begin{array}{ccc}0&1&0\\-1&0&0\\0&0&0\end{array}\right]
\end{align*}
We first explain how the matrices of type $C$ above relate to the matrices of type $B$ which we first arrived at, and then we will explain 
how the matrices of type $C$ deform, which will justify the new, less obvious stratification of the moduli space.
It is evident that some important changes have occurred in the alignment of the strata. First, let us discuss the easy part. The matrix $C_2$ is clearly equivalent
to the matrix $B_3$, while the matrix $C_6$ is equivalent to the matrix $B_4$.  The family $C_5(p:q)$ is equivalent pointwise to the family $B_2(p:q)$. The family
$C_1(p:q)$ mostly coincides with the family $B_1(p:q)$, with the exception that $C_1(1:1)\not\sim B_1(1:1)$ and $C_1(0:0)\not\sim B_1(0:0)$.
In fact $C_1(1:1)\sim B_6$, which is interesting because $B_6$ is indecomposable.  
Secondly, as we shall discuss later, $C_1(0:0)\sim C_5(1:\gamma)$ where $\gamma$ is an arbitrary primitive
sixth root of unity. On the other hand $B_1(1:1)\sim C_4$, while $B_1(0:0)\sim C_5(1:1)$.  Finally $C_3\sim B_5$.

It may seem strange that only one of the indecomposable matrices gives a separate stratum, while the other one is part of a family. However, it turns out that
decomposability/indecomposability is not preserved under deformations.  The stratification above has the important deformation property, that elements either
deform smoothly along a family, or have jump deformations to another stratum.

Another surprising property is that the generic element in the family $C_1(p:q)$ is really just an ordinary element in another family.  Although this is the first
example we have encountered in this paper, in other classification problems we have studied,
in particular in the study of moduli spaces of algebras, this kind of phenomenon occurs
frequently. In fact, the family $B_1(p:q)$ had a similar overlap in that $B_1(0:0)\sim B_2(1:1)$. It may also seem strange that the element to which $C_1(0:0)$ is
equivalent is $C_5(1:\gamma)$ where $\gamma$ is any primitive sixth root of unity. But this follows from the fact that the families $C_1$ and $C_5$ are parametrized
by $\P^1/\Sigma_2$, so we have $C_5(1:\gamma)\sim C_5(\gamma^5:1)\sim C_5(1:\gamma^5)$, where the middle equivalence arises by multiplying both elements by $\gamma^5$,
and using the fact that $\gamma^6=1$. Since $\gamma$ and $\gamma^5$ are the two primitive sixth roots of unity, this explains why such a number could arise.

\begin{figure}[ht]

\begin{tikzpicture}[]
\node[algebra] at (6,7.5)      (C500)    {\bm{$C_5\of{0:0}$}};
\node[algebra] at (2,6.5)      (C6)    {\bm{$C_6$}};
\node[algebra] at (6,5.5)      (C5pq)    {\bm{$C_5\of{p:q}$}};
\node[algebra] at (6,1)      (C1pq)    {\bm{$C_1\of{p:q}$}};
\node[algebra] at (6,3.5)    (C3)      {\bm{$C_3$}};
\node[algebra] at (1.5,3)  (C2)     {\bm{$C_2$}};
\node[algebra] at (9.5,3)  (C4)     {\bm{$C_4$}};
\node[algebra] at (3.5,3.5)  (C100)    {\bm{$C_1\of{0:0}$}};
\node[algebra] at (4.65,5.25)   (C51-1)   {\footnotesize\bm{$\of{1:-1}$}};
\node[algebra] at (7.5,5.25)   (C511)   {\footnotesize\bm{$\of{1:1}$}};
\node[algebra] at (4,0.6)   (C11-1)   {\footnotesize\bm{$\of{1:-1}$}};
\node[algebra] at (8,0.6)   (C111)   {\footnotesize\bm{$\of{1:1}$}};
\draw[family]         (C5pq)+(0,-0.25)   ellipse (2cm and 0.66cm);
\draw[family]         (C1pq)+(0,-0.25)   ellipse (3cm and 1cm);
\draw[jumpFamily]     (C500) -- (4,5.35);
\draw[jumpFamily]     (C500) -- (8,5.35);
\draw[jumpFamily]     (C3) -- (4.15,5);
\draw[jumpFamily]     (C3) -- (7.85,5);
\draw[jumpFamily]     (C3) -- (3.2,1.1);
\draw[jumpFamily]     (C3) -- (8.8,1.1);
\draw[jumpFamily]     (C100) -- (3,0.8);
\draw[jumpFamily]     (C100) -- (8,1.5);
\draw[jumpSingle]     (C6) -- (C51-1);
\draw[jumpSingle]     (C6) -- (C2);
\draw[jumpSingle]     (C2) -- (C11-1);
\draw[jumpSingle]     (C511) -- (C4);
\draw[jumpSingle]     (C4) -- (C111);
\end{tikzpicture}
\caption{The Moduli Space of Bilinear Forms on $\C^3$}
\end{figure}
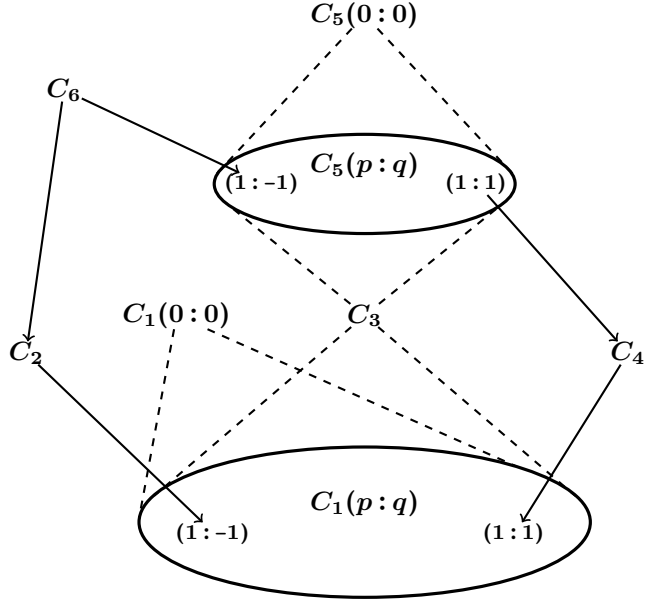

Now we will discuss the deformations of the matrices. First, we note that the numbering has been chosen in such a manner that an element in $C_k$ will only deform along the family $C_k$ (if it is a family) and to elements in families $C_\ell$ for $\ell<k$. The fact that such a numbering is possible is related to the property that if $A$ has a jump deformation to $B$, it cannot happen that $B$ jumps to $A$. A simple way to see this is that the cohomology of $B$ will have smaller dimension than that of $A$, when there is a jump from $A$ to $B$.

Elements in the family $C_1(p:q)$ other than the generic element $C_1(0:0)$ deform only along the family. Since the generic element is equivalent to an ordinary element in the family $C_5(p:q)$, we see from that fact, that the corresponding element $C_2(1,\gamma)$ has jump deformations to every member of the family $C_1(p:q)$ except $C_1(0:0)$, because that is the deformation picture for $C_1(0:0)$.  The generic element in a family always has
jump deformations to every other element in the family.

The element $C_2$ has a jump deformation to $C_1(1: -1)$ and smooth deformations in a nbd of this point. The element $C_3$ has jump deformations to every element in
the family $C_1(p:q)$ except $C_1(0: 0)$.  The element $C_4$ has a jump deformation to $C_1(1:1)$ and smooth deformations in a nbd of this point.

The element $C_5(x:y)$ has jump deformations to every element in the family $C_1(p:q)$ except $C_1(0:0)$. Since $C_1(0:0)$ actually is an ordinary element of the
family $C_5(p:q)$, we don't expect an element $C_5(x:y)$ to jump to it (except of course for the generic element $C_5(0:0)$). 
Also every member of the family $C_5(p:q)$ has a jump deformation
to $C_3$.

On the other hand, only $C_5(1:1)$ and $C_5(0:0)$ have deformations to $C_4$. This is interesting, because the only symmetric bilinear forms are $C_4$, $C_5(1:1)$ and
$C_5(0:0)$. It is easy to see that an element which is not symmetric cannot have a jump deformation to a symmetric element, so the fact that only two elements of the
family $C_5(p:q)$ have jump deformations to $C_4$ is consistent with that observation.

The elements $C_5(1:-1)$ and $C_5(0: 0)$ have jump deformations to $C_2$. Moreover, $C_5(0:0)$ has jump deformations to all other elements of the family $C_5(p:q)$.
Also, $C_5(x:y)$ has smooth deformations along the family.  In fact, there are also smooth deformations of $C_5(x:y)$ in a nbd of any point in $C_1(p:q)$ except $C_1(0:0)$, because
whenever there is a jump deformation to an element of a family, there are smooth deformations in a nbd of that element.

Finally, we address deformations of the element $C_6$.  Note that it is the only antisymmetric element in our space, so it would be impossible for any other element
to have a jump deformation to it.  There is a jump deformation from $C_6$ to $C_5(1:-1)$ and smooth deformations in a nbd of this point, but no other jumps to this family.  Of course, because $C_5(1:-1)$ jumps to $C_2$, so does $C_6$, because of the transitivity of jump deformations, meaning that if $A$ jumps to $B$ and $B$ jumps to $C$, then $A$ must jump to $C$.  There is also a jump deformation from $C_6$ to $C_3$, but not to $C_4$ for reasons listed above. We also have that $C_6$ has
jump deformations to every member of the family $C_1(p:q)$ except $C_1(0:0)$. Finally, $C_6$ has a jump deformation to $C_2$.

\subsubsection{Justification for the new decomposition}

First, we explain the difference between the family $B_1(p:q)$ and $C_1(p:q)$. The shifting of the element corresponding to $(1:1)$ is important.
The two elements in question are $B_1(1:1)=C_4$, while $C_1(1:1)=B_6$. In fact, we saw that both $C_4$ and $C_1(1:1)$ deform in a nbd of the point $C_1(1:1)$, so the
question is which element really should belong to the family.  The answer is given by deformation theory.  Since $C_4$ has a jump deformation to $C_1(1:1)$, not the
other way around, it is $C_4$ which does not belong to the family. In order to find a projective parametrization that included the element which was $B_6$ in the family
$C_1(p:q)$, it was necessary for the generic element to be shuffled.  In fact, every element of the family $C_5(x:y)$ (except $C_5(0:0)$) behaves like the generic element in the family $C_1(p:q)$ would behave, in the sense that they all have jump deformations to the elements of the family $C_1(p:q)$ (except $C_1(0:0)$).

It is important to note that in our decomposition, it was possible to give a numbering convention in such a fashion that elements had jump deformations to elements whose
index was smaller (except for the $(0:0)$ case).  This has been one of the strategies we have been developing in numbering the algebras in our study of moduli spaces,
and it seems a good strategy here as well.

\subsubsection{Computational Techniques}

Let us explain how we obtained the deformations of our elements. First, if the element is represented by a matrix $D$, then we can obtain a matrix
$$
D_t=D+\sum t_{i,j}e_{i,j},
$$
which is the most general linear deformation possible.  In most cases, we set some of the parameters $t_{i,j}$ to be zero, and label them in a simpler way.
Then suppose we want to find a deformation $D_t$ between $D$ and some other element represented by the matrix $B$. We let $G=(g_{i,j})$ be a generic square
matrix (of the right size), and consider the equation
$$
G^T D_t G=B.
$$
We solve the associated system of equations for  the variables $g_{i,j}$ and the other parameters which may occur if $D$ or $B$ are families.  Then we determine
which solutions give rise to matrices whose determinants are not identically zero.  Afterwards, we study the solutions to see what values of the family parameters
are forced to take, and whether the solution is local in the sense that every nbd of the origin in the $t$-space has at least some element in the solution.

In practice, this is pretty easy to implement with a computer algebra system, but even for the $3$-dimensional complex case, if we don't restrict the number of nonzero
parameters $t_{i,j}$, it is not easy to solve on the computer, so it can be a bit tricky to piece out the deformations.

\subsection{Comparison to the Horn-Sergeichuk decomposition}
One advantage of the decomposition given by Horn and Sergeichuk is that it works for all fields and all dimensions.  On the other hand, the decomposition we have given
has a nice relation to the geometry of the complex moduli space, that results and gives a stratification by projective orbifolds of a very simple type. It is also
relatively easy to give a complete picture of the deformations of the elements in the space using our decomposition.
However, we have not determined a general description for the decomposition of
moduli spaces of bilinear forms, as we were able to do for the moduli space of matrices under similarity transformations.

 \bibliographystyle{amsplain}
\providecommand{\bysame}{\leavevmode\hbox to3em{\hrulefill}\thinspace}
\providecommand{\MR}{\relax\ifhmode\unskip\space\fi MR }
\providecommand{\MRhref}[2]{%
  \href{http://www.ams.org/mathscinet-getitem?mr=#1}{#2}
}
\providecommand{\href}[2]{#2}

\end{document}